\documentclass[12pt]{amsart}

\usepackage{graphicx}

\usepackage{subcaption} 
\usepackage{xfrac}   
\usepackage{faktor}
\usepackage{tikz-cd}
\usepackage{mathrsfs}
\usepackage{mathtools}
\usepackage{hyperref}
\usepackage{amsmath}
\usepackage{amssymb}
\hypersetup{bookmarks=true,
	unicode=true,
	colorlinks=true,
	citecolor=black,
	linkcolor=black,
	urlcolor=black,
	plainpages=false,
	pdfpagelabels=true}

 \usepackage{url}	
 \allowdisplaybreaks 

\usepackage{tikz-cd}
\usepackage{pgf}

\usepackage{xcolor}

\usepackage{comment} 

\usepackage[all]{xy}

\newtheorem{teo}{Theorem}[section]
\newtheorem{theorem}[teo]{Theorem}

\newtheorem{corollary}[teo]{Corollary}

\newtheorem{lemma}[teo]{Lemma}

\newtheorem{proposition}[teo]{Proposition}

\theoremstyle{definition}
\newtheorem{definition}[teo]{Definition}

\newtheorem{example}[teo]{Example}

\theoremstyle{remark}
\newtheorem{remark}[teo]{Remark}

\numberwithin{figure}{section}

\newcommand{\Ba}{\mathcal{B}}

\newcommand{\CC}{\mathcal{C}}

\newcommand{\Dc}{\mathcal{D}}

\newcommand{\E}{\mathcal{E}}

\newcommand{\op}{\mathrm{op}}

\newcommand{\Sets}{\mathrm{Set}}

\newcommand{\Uc}{\mathcal{U}}

\DeclareMathOperator{\Hom}{Hom}

\newcommand{\Cat}{\mathrm{Cat}}

\begin{document}

\title[Vector fields, initial scaffolds and database reduction]{Vector fields, initial scaffolds and database reduction}
\thanks{}

\author[Isaac Carcacía-Campos]{%
	Isaac Carcacía-Campos
}

 \address{%
           	Isaac Carcacía-Campos
            \\
              Departamento de Matem\'aticas, Universidade de Santiago de Compostela, 15782-SPAIN}
               \email{isaac.c.campos@usc.es}


\begin{abstract} 
Reduction replaces a mathematical object with a simpler model that
retains the relevant information. We introduce left and right vector fields on small categories as tools for reducing finite acyclic categories while preserving their directed homotopical information. We relate these fields to directed deformation retracts and beat-object reductions, and show that right vector-field reductions preserve the directed sectional category of right directed fibrations and the global sections of functorial databases.

We also extend initial scaffolds from posets to acyclic categories. These provide smaller indexing categories that preserve limits and, in particular, globally coherent selections in databases. Finally, we prove that initial scaffolds are preserved by right directed deformation retracts and hence by right vector-field reductions.
\end{abstract}


\maketitle


\section*{Introduction}
A recurring strategy in mathematics and science is to replace a complicated object by a simpler model that retains the information relevant to the problem at hand. In combinatorial topology, this principle underlies the reduction of finite \(T_0\)-spaces and finite posets by beat points \cite{Stong,MR3024764}, the theory of elementary simplicial collapses \cite{WhiteheadCollapse}, and discrete Morse theory for cell complexes and finite posets \cite{Forman,Scoville,MosqueraPosets}. This viewpoint is also related to the dynamics of finite spaces, where semiflows are closely connected with down beat points and deformation retracts \cite{ChocanoSemiflows,ChocanoFlowsAlexandroff}. The purpose of this article is to develop analogous reduction procedures for small categories, especially finite acyclic ones, which can be seen as generalizations of posets.

From this perspective, the homotopy theory of finite posets lies inside a broader homotopy theory of small categories. Natural transformations have long been interpreted as building blocks for homotopies between functors \cite{LEEHomotopy}. Since natural transformations are intrinsically oriented, they also lead to the directed homotopy theory studied by Grandis \cite{GrandisShape} and by the author in a recent article \cite{DirectedHomotopyDatabases}. Our aim is to develop categorical reductions that preserve either directed homotopical information or the limits of diagrams indexed by the category, with particular attention to functorial databases.

The first construction is based on categorical vector fields. A vector is often first introduced geometrically as an arrow before being replaced by the abstract notion of an element of a vector space. We return to this elementary picture and regard morphisms as vectors. The arrows starting or ending at an object form two indexed families of comma categories, whose Grothendieck constructions define the corresponding tangent categories. Vector fields are sections of their canonical projections. Equivalently, they are endofunctors equipped with natural transformations to or from the
identity.

Our notion differs in particular from the discrete Morse-theoretic vector fields used by d'Antonio and Delucchi for acyclic categories \cite{dAntonioDelucchiToricArrangements}. Their construction is based on partial acyclic matchings and is designed to produce cellular collapses and Morse complexes.

For finite acyclic categories, every vector field stabilizes to a directed retract onto its fixed subcategory. Building on Tanaka's reduction theory for acyclic categories \cite{StrongHomTan,SimpleHom_Tan}, we show that this retract decomposes into down or up beat-object removals, according to the orientation of the field, and that every such oriented reduction determines an idempotent vector field. The related work of Chocano on semiflows of finite spaces \cite{ChocanoFlowsAlexandroff,ChocanoSemiflows} also motivates our categorical flows of successive vector-field
reductions

The homotopical nature of vector-field reductions makes them compatible with lifting problems. Building on the strong and directed fibrations studied in \cite{carcaciacampos2026weakstrongfibrationsfunctors,
DirectedHomotopyDatabases}, we prove that restriction along a vector-field retract preserves the directed sectional category of every right directed fibration. Hence the sectional problem may be transferred to the fixed subcategory.

In the functorial approach to databases, schemas are small categories and instances are set-valued functors \cite{SevenSketches,SpivakDataMigration,SpivakDatabaseQueries}. We prove that vector-field reductions preserve their global sections. For Grothendieck opfibrations, they also preserve ordinary sectional category.

A second, complementary reduction is provided by initial scaffolds. Dey and Lesnick introduced initial scaffolds for posets with finite downsets as explicit models for minimal initial functors \cite{dey2026limitcomputationposetsminimal}. Their principal applications concern the efficient computation of limits of vector-space-valued diagrams. We extend their construction to lower well-founded acyclic categories by replacing open downsets with punctured lower comma categories. Our main application is instead to set-valued diagrams arising as functorial databases.

The resulting scaffold is a subcategory whose inclusion is initial. It therefore preserves the limits of all diagrams and, in particular, the set of global sections of every database instance. We also prove that the removal of a down beat object leaves the initial scaffolds unchanged. Since every right vector field stabilizes through a sequence of down beat-object removals, a finite acyclic category and the fixed subcategory of any right vector field have exactly the same initial scaffolds. 

The paper is organized as follows. Sections~\ref{sec:tangent_category}--\ref{sec:homotopy_retracts_vector_fields} develop categorical vector fields, their fixed subcategories, and their relation with directed retracts and beat-object reductions. Section~\ref{sec:homotopy_fluxes} recalls strong and directed homotopies and introduces categorical flows as successive vector-field reductions preserving strong homotopy type.  Section~\ref{sec:sectional_category_reductions} studies the preservation of directed sectional category. Section~\ref{sec:databases_retractions} applies these results to Grothendieck opfibrations and functorial databases. Finally, Section~\ref{sec:initial-scaffolds} extends initial scaffolds to acyclic categories and relates them to beat-object and vector-field reductions.

\section{Tangent categories}
\label{sec:tangent_category}

Let \(\mathbb I_1=(0\to1)\) denote the walking arrow. The arrow category
of a small category \(\CC\) is the functor category
\(\CC^{\mathbb I_1}\). Its objects are the morphisms of \(\CC\), and its
morphisms are commutative squares. It has two canonical projections
\[
\operatorname{dom},\operatorname{cod}
\colon
\CC^{\mathbb I_1}\longrightarrow\CC.
\]

For \(c\in\CC\), the category \(c\downarrow\CC\) consists of the arrows
starting at \(c\), whereas \(\CC\downarrow c\) consists of the arrows
ending at \(c\). These assignments define functors
\[
-\downarrow\CC\colon\CC^{\op}\longrightarrow\Cat,
\qquad
\CC\downarrow-\colon\CC\longrightarrow\Cat
\]
by composing or precomposing the morphisms.

We briefly recall the two Grothendieck constructions used below. If
\(F\colon\CC\to\Cat\) is covariant, the objects of
\(\int^\CC F\) are pairs \((c,x)\), with \(x\in F(c)\). A morphism
\[
(c,x)\longrightarrow(d,y)
\]
is a pair \((f,\varphi)\), where \(f\colon c\to d\) and
\[
\varphi\colon F(f)(x)\longrightarrow y
\]
is a morphism in \(F(d)\).

Dually, if \(G\colon\CC^{\op}\to\Cat\), the objects of
\(\int_\CC G\) are pairs \((c,x)\), with \(x\in G(c)\), and a morphism
\[
(c,x)\longrightarrow(d,y)
\]
is a pair \((f,\varphi)\), where \(f\colon c\to d\) and
\[
\varphi\colon x\longrightarrow G(f)(y)
\]
is a morphism in \(G(c)\). In both cases, projection onto the first
coordinate defines a functor to \(\CC\).

\begin{definition}\label{def:tangent_categories}
The \emph{right tangent category} and the \emph{left tangent category} of
\(\CC\) are, respectively,
\[
T_R\CC
=
\int^\CC(\CC\downarrow-),
\qquad
T_L\CC
=
\int_\CC(-\downarrow\CC),
\]
with their canonical projections
\[
\pi_R\colon T_R\CC\longrightarrow\CC,
\qquad
\pi_L\colon T_L\CC\longrightarrow\CC.
\]
\end{definition}

The fibre of \(\pi_R\) over \(c\) is \(\CC\downarrow c\), while the fibre
of \(\pi_L\) over \(c\) is \(c\downarrow\CC\). Thus, the two projections
collect, respectively, the arrows ending and starting at each object.

\begin{proposition}\label{prop:tangent_arrow_category}
There are canonical isomorphisms
\[
T_R\CC\cong\CC^{\mathbb I_1}\cong T_L\CC.
\]
Under these identifications,
\[
\pi_R=\operatorname{cod}
\qquad\text{and}\qquad
\pi_L=\operatorname{dom}.
\]
\end{proposition}
\begin{proof}
An object of either tangent category is simply a morphism of \(\CC\).

Consider \(u\colon x\to c\) and \(v\colon y\to d\) as objects of
\(T_R\CC\). A morphism between them consists of morphisms
\[
h\colon x\longrightarrow y,
\qquad
f\colon c\longrightarrow d
\]
satisfying
\[
fu=vh.
\]
Thus, it is precisely a commutative square from \(u\) to \(v\), and hence
a morphism in \(\CC^{\mathbb I_1}\). This identification preserves
identities and composition and sends \(\pi_R\) to
\(\operatorname{cod}\).

The contravariant construction gives the analogous description of
\(T_L\CC\): its morphisms are again commutative squares, and its
projection records the domain.
\end{proof}

\begin{remark}
Although \(T_R\CC\) and \(T_L\CC\) have canonically isomorphic total
categories, their projections are different functors.
\end{remark}

\section{Vector fields on small categories}
\label{sec:vector_fields}

The two tangent projections give rise to two notions of vector field,
according to which endpoint is retained by the tangent projection.

\begin{definition}\label{def:categorical_vector_fields}
Let \(\CC\) be a small category.

\begin{enumerate}
    \item A \emph{right vector field} on \(\CC\) is a section
    \(
    V\colon\CC\longrightarrow T_R\CC
    \)
    of \(\pi_R\).

    \item A \emph{left vector field} on \(\CC\) is a section
    \(
    V\colon\CC\longrightarrow T_L\CC
    \)
    of \(\pi_L\).
\end{enumerate}
\end{definition}

Using the identifications
\(T_R\CC\cong\CC^{\mathbb I_1}\cong T_L\CC\), vector fields admit the
following endofunctorial description.

\begin{proposition}\label{prop:vector_fields_natural_transformations}
Let \(\CC\) be a small category.

\begin{enumerate}
    \item A right vector field is equivalently an endofunctor
    \(R\colon\CC\to\CC\) equipped with a natural transformation
    \[
    \varepsilon\colon R\Rightarrow1_\CC.
    \]

    \item A left vector field is equivalently an endofunctor
    \(L\colon\CC\to\CC\) equipped with a natural transformation
    \[
    \eta\colon1_\CC\Rightarrow L.
    \]
\end{enumerate}
\end{proposition}

\begin{proof}
Let \(V\colon\CC\to T_R\CC\) be a right vector field. Since
\(\pi_RV=1_\CC\), for every \(c\in\CC\) the object \(V(c)\) is an arrow
\[
\varepsilon_c\colon R(c)\longrightarrow c.
\]
For \(f\colon c\to d\), the morphism \(V(f)\) is a commutative square
\[
\begin{tikzcd}
R(c) \arrow[r, "\varepsilon_c"] \arrow[d, "R(f)"']
&
c \arrow[d, "f"]
\\
R(d) \arrow[r, "\varepsilon_d"']
&
d.
\end{tikzcd}
\]
Thus, \(R\) is an endofunctor and
\(\varepsilon\colon R\Rightarrow1_\CC\) is natural. Conversely, any such
transformation defines a section of \(\pi_R\). The left case is dual.
\end{proof}

Thus, right and left vector fields describe directed \textit{deformations} of the
identity in opposite orientations.

\subsection{Fixed objects}
\label{subsec:fixed_objects}

\begin{definition}\label{def:fixed_object}
Let \(V\) be a right or left vector field on \(\CC\). An object
\(c\in\CC\) is \emph{fixed by \(V\)} if
\[
V(c)=1_c
\]
in the corresponding tangent category.

Equivalently, if \(V=(R,\varepsilon)\) is right, then
\(R(c)=c\) and \(\varepsilon_c=1_c\); if \(V=(L,\eta)\) is left, then
\(L(c)=c\) and \(\eta_c=1_c\).
\end{definition}

We write \(\operatorname{Fix}(V)\) for the full subcategory determined by
the fixed objects. When the vector field is represented by
\((R,\varepsilon)\) or \((L,\eta)\) we also write
\(\operatorname{Fix}(R)\) or \(\operatorname{Fix}(L)\), respectively.

\begin{proposition}\label{prop:fixed_subcategory}
The endofunctor associated with a vector field restricts to the identity
on its fixed subcategory.
\end{proposition}

\begin{proof}
Suppose first that \(V=(R,\varepsilon)\) is a right vector field. If
\(f\colon c\to d\) is a morphism between fixed objects, naturality gives
\[
f\varepsilon_c=\varepsilon_dR(f).
\]
Since \(\varepsilon_c=1_c\) and \(\varepsilon_d=1_d\), we obtain
\(R(f)=f\). The left case is dual.
\end{proof}

\subsection{Examples}

\begin{example}\label{ex:identity_vector_field}
The identity transformation of \(1_\CC\) defines both a right and a left
vector field, and every object is fixed.
\end{example}

\begin{example}\label{ex:initial_vector_field}
If \(\CC\) has an initial object \(c_0\), the unique morphisms
\(c_0\to c\) define a right vector field
\[
\overline{c_0}\Rightarrow1_\CC.
\]
Dually, a terminal object \(c_1\) determines a left vector field
\[
1_\CC\Rightarrow\overline{c_1}.
\]
\end{example}

\begin{example}\label{ex:simple_vector_field}
Let \(\mathcal D\) be the free category generated by
\[
\begin{tikzcd}[column sep=3em]
a
  \arrow[r, "h"]
&
b
  \arrow[r, "f", bend left=35]
  \arrow[r, "g"', bend right=35]
&
c.
\end{tikzcd}
\]
In particular, \(fh\) and \(gh\) are distinct morphisms from \(a\) to
\(c\).

Define \(R\colon\mathcal D\to\mathcal D\) by
\[
R(a)=R(b)=a, \text{ }
R(c)=c,
\text{ }
R(h)=1_a,
\text{ }
R(f)=fh,
\text{ }
R(g)=gh.
\]
The components
\[
\varepsilon_a=1_a,
\qquad
\varepsilon_b=h,
\qquad
\varepsilon_c=1_c
\]
define a natural transformation
\(\varepsilon\colon R\Rightarrow1_{\mathcal D}\).

Its fixed objects are \(a\) and \(c\). Hence
\[
\operatorname{Fix}(R)
=
\mathcal D|_{\{a,c\}},
\]
and this full subcategory contains the distinct parallel morphisms
\(fh,gh\colon \allowbreak a\to c\).
\end{example}

\section{Retracts and vector fields}
\label{sec:homotopy_retracts_vector_fields}

A vector field describes a one-step directed deformation of the identity,
but it need not define a retraction. We shall show
that, for finite acyclic categories, iteration eventually produces a
retract with suitable directed homotopical properties. We first introduce
the two notions needed to make this correspondence precise.

A full subcategory \(\Uc\subseteq\CC\) is a \emph{right directed
deformation retract} of \(\CC\) if there are functors
\(
\iota\colon\Uc\hookrightarrow\CC,
\text{ }
r\colon\CC\longrightarrow\Uc
\)
such that \(r\iota=1_\Uc\), together with a natural transformation
\(
\alpha\colon\iota r\Rightarrow1_\CC
\)
whose restriction to \(\Uc\) is the identity. A \emph{left directed
deformation retract} is defined dually.

A right vector field \((R,\varepsilon)\) is \emph{idempotent} if
\(R^2=R\) and dually for a left vector field.

Finally, recall that a category \(\CC\) is \emph{acyclic} \cite[Chapter 10]{Kozlov} if every endomorphism is an identity
and, for distinct objects \(c,d\in\CC\), the existence of a morphism
\(c\to d\) implies that there is no morphism \(d\to c\).

\begin{lemma}\label{lem:idempotent_field_normalized}
Let \(\CC\) be acyclic and let
\((R,\varepsilon)\) be an idempotent right vector field. Then
\[
\varepsilon R=R\varepsilon=1_R.
\]
In particular, every object in the image of \(R\) is fixed.
\end{lemma}

\begin{proof}
For each \(c\in\CC\), both
\[
\varepsilon_{R(c)},R(\varepsilon_c)\colon R(c)\longrightarrow R(c)
\]
are endomorphisms. Since \(\CC\) is acyclic, both are identities.
\end{proof}

\begin{proposition}\label{prop:idempotent_fields_retracts}
Let \(\CC\) be acyclic. Idempotent right vector fields on \(\CC\)
correspond to right directed deformation retracts onto full
subcategories. Dually, idempotent left vector fields correspond to left
directed deformation retracts.
\end{proposition}

\begin{proof}
Let \((R,\varepsilon)\) be an idempotent right vector field. By
Lemma~\ref{lem:idempotent_field_normalized}, \(R(c)\) is fixed for every
\(c\), so \(R\) defines by corestriction a functor
\[
r\colon\CC\longrightarrow\operatorname{Fix}(R).
\]
If \(\iota\) denotes the inclusion, then
\(R=\iota r\). Moreover, Proposition~\ref{prop:fixed_subcategory} gives
\(r\iota=1_{\operatorname{Fix}(R)}\), and \(\varepsilon\) becomes a
natural transformation
\[
\iota r\Rightarrow1_\CC
\]
relative to \(\operatorname{Fix}(R)\).

Conversely, a right directed deformation retract
\[
\Uc
\mathrel{\substack{\xrightarrow{\iota}\\[-1mm]\xleftarrow[r]{}}}
\CC
\]
defines the endofunctor \(R=\iota r\). Since \(r\iota=1_\Uc\),
\[
R^2=\iota r\iota r=\iota r=R,
\]
and the transformation \(\iota r\Rightarrow1_\CC\) makes \(R\) an
idempotent right vector field. The left statement is dual.
\end{proof}

Thus, directed retracts correspond to vector fields that have already
stabilized. Different vector fields may nevertheless stabilize to the
same retract.

For a right vector field \((R,\varepsilon)\) and \(n\geq1\), define
\[
\varepsilon^{(n)}\colon R^n\Rightarrow1_\CC
\]
recursively by
\[
\varepsilon^{(1)}=\varepsilon,
\qquad
\varepsilon^{(n+1)}
=
\varepsilon^{(n)}\circ(\varepsilon R^n).
\]
Its component at \(c\) is
\[
R^n(c)
\xrightarrow{\varepsilon_{R^{n-1}(c)}}
R^{n-1}(c)
\longrightarrow\cdots\longrightarrow
R(c)
\xrightarrow{\varepsilon_c}
c.
\]

\begin{theorem}\label{thm:vector_field_stabilization}
Let \(\CC\) be a finite acyclic category and let
\((R,\varepsilon)\) be a right vector field. There exists \(N\geq1\) such
that \(R^N\) is idempotent, takes values in
\(\operatorname{Fix}(R)\), and determines a right directed deformation
retract onto \(\operatorname{Fix}(R)\).
\end{theorem}

\begin{proof}
If \(c\) is not fixed, then \(\varepsilon_c\colon R(c)\to c\) is
non-identity: otherwise \(R(c)=c\), and acyclicity would force
\(\varepsilon_c=1_c\). Thus, until a fixed object is reached, iteration
produces a strictly descending chain
\[
\cdots\longrightarrow R^2(c)\longrightarrow R(c)\longrightarrow c.
\]
No object can occur twice, since that would produce a directed cycle.
Finiteness therefore implies that every such chain reaches a fixed
object, and their lengths admit a common bound \(N\geq1\).

Hence \(R^N(c)\) is fixed for every \(c\). Since \(R\) restricts to the
identity on its fixed subcategory,
\[
R^{N+1}=R^N.
\]
It follows inductively that \(R^{N+k}=R^N\) for every \(k\geq0\).
Taking \(k=N\), we obtain \((R^N)^2=R^N\).

It remains to identify the fixed objects. Clearly,
\(\operatorname{Fix}(R)\subseteq\operatorname{Fix}(R^N)\). Conversely,
if \(R^N(c)=c\), the iterated transformation gives a directed cycle from
\(c\) to itself. Acyclicity forces every morphism in this cycle to be an
identity, so \(c\) is already fixed by \(R\). Thus
\[
\operatorname{Fix}(R^N)=\operatorname{Fix}(R).
\]
The conclusion now follows from
Proposition~\ref{prop:idempotent_fields_retracts}, applied to
\(\varepsilon^{(N)}\colon R^N\Rightarrow1_\CC\).
\end{proof}

\subsection{Beat objects and vector fields}
\label{subsec:beat_objects_vector_fields}

The reduction of finite \(T_0\)-spaces by beat points goes back to Stong
\cite{Stong}, see also \cite{MR3024764} for a broader treatment. Tanaka extended this theory to finite acyclic categories \cite{StrongHomTan,SimpleHom_Tan}.

For \(c\in\CC\), let
\[
\partial(\CC\downarrow c)
=
(\CC\downarrow c)\setminus\{(c,1_c)\},
\]
and define \(\partial(c\downarrow\CC)\) dually. An object \(c\) is a
\emph{down beat object} if \(\partial(\CC\downarrow c)\) has a terminal
object, and an \emph{up beat object} if
\(\partial(c\downarrow\CC)\) has an initial object.

Equivalently, \(c\) is down beat if there is a morphism
\(u\colon b\to c\) through which every non-identity morphism into \(c\)
factors uniquely. For posets, these definitions recover the usual down
and up beat points.

\subsubsection{Beat-object removals as vector fields}

\begin{proposition}\label{prop:beat_object_vector_field}
Let \(\CC\) be acyclic.

\begin{enumerate}
    \item Removing a down beat object produces a right directed deformation
    retract and hence an idempotent right vector field.
    \item Removing an up beat object produces a left directed
    deformation retract and hence an idempotent left vector field.
\end{enumerate}
\end{proposition}

\begin{proof}
Let \(c\) be down beat, and let \(u\colon b\to c\) be terminal in
\(\partial(\CC\downarrow c)\). Define
\[
r\colon\CC\longrightarrow\CC\setminus\{c\}
\]
by \(r(c)=b\) and \(r(x)=x\) for \(x\neq c\). If \(v\colon x\to c\) is
non-identity, set \(r(v)=\widetilde v\), where
\(v=u\widetilde v\) is its unique factorization; if \(w\colon c\to y\),
set \(r(w)=wu\). On all other morphisms, let \(r\) act as the identity.

The universal property of \(u\) gives functoriality. If \(\iota\) is the
inclusion, then \(r\iota=1\), and the components
\[
\alpha_c=u,
\qquad
\alpha_x=1_x\quad(x\neq c)
\]
define a natural transformation
\(\iota r\Rightarrow1_\CC\). The conclusion follows from
Proposition~\ref{prop:idempotent_fields_retracts}. The up beat case is
dual.
\end{proof}

\subsubsection{Vector fields as beat-object reductions}

\begin{lemma}\label{lem:minimal_nonfixed_down_beat}
Let \(\alpha\colon R\Rightarrow1_\CC\) be an idempotent right vector field
on a finite acyclic category. Every minimal non-fixed object \(c\) is down
beat, with terminal morphism
\[
\alpha_c\colon R(c)\longrightarrow c.
\]
\end{lemma}

\begin{proof}
Let \(g\colon x\to c\) be non-identity. Then \(x<_\CC c\), so the
minimality of \(c\) implies that \(x\) is fixed. Naturality therefore gives
\(
g=\alpha_cR(g),
\)
so \(g\) factors through \(\alpha_c\).

If \(\alpha_ch=g\), applying \(R\) yields
\(
R(h)=R(g),
\)
because \(R(\alpha_c)=1_{R(c)}\). Since \(x\) and \(R(c)\) are fixed,
\(R(h)=h\), and hence \(h=R(g)\). Thus, the factorization is unique.
\end{proof}

\begin{theorem}\label{thm:right_field_down_beat_reduction}
Let \(\CC\) be a finite acyclic category. The fixed subcategory of every
right vector field is obtained by finitely many down beat-object removals.
\end{theorem}

\begin{proof}
Stabilize the field by
Theorem~\ref{thm:vector_field_stabilization}. If a non-fixed object
remains, choose one that is minimal. By
Lemma~\ref{lem:minimal_nonfixed_down_beat}, it is down beat.

Remove it and restrict the stabilized field to the remaining full
subcategory. Its image still lies in the fixed subcategory, so the process
may be repeated. Finiteness gives a sequence
\[
\CC=\CC_0\searrow\CC_1\searrow\cdots\searrow\CC_n
=
\operatorname{Fix}(R)
\]
of down beat-object removals.
\end{proof}

\begin{theorem}\label{thm:left_field_up_beat_reduction}
Let \(\CC\) be a finite acyclic category. The fixed subcategory of every
left vector field is obtained by finitely many up beat-object removals.
\end{theorem}

\begin{proof}
Apply Theorem~\ref{thm:right_field_down_beat_reduction} to
\(\CC^{\op}\).
\end{proof}

\begin{corollary}\label{cor:vector_fields_beat_reductions}
On a finite acyclic category, stabilized right and left vector fields are
realized by sequences of down and up beat-object removals, respectively.
Conversely, every such oriented sequence determines an idempotent vector
field.
\end{corollary}

\begin{remark}\label{rem:fields_more_than_beat_sequence}
This correspondence is not bijective: different vector fields may
stabilize to the same retract, and a retract may admit several
factorizations by beat-object removals. A vector field records the
reduction functorially through a single natural transformation, whereas a
beat-object sequence records one factorization of the stabilized retract.
\end{remark}

\subsection{An example}
\label{subsec:complete_vector_field_example}

\begin{example}\label{ex:beat-point-vector-field}
Let \(\mathcal D\) be the category generated by
\[
\begin{tikzcd}[column sep=3em]
p \arrow[r, "k"]
&
a \arrow[r, "h"]
&
b
  \arrow[r, "f_1", bend left=45]
  \arrow[r, "f_2" description]
  \arrow[r, "f_3"', bend right=45]
&
c
\end{tikzcd}
\]
subject to \(f_1h=f_3h\). The objects \(a\) and \(b\) are down beat, with
terminal morphisms \(k\colon p\to a\) and \(h\colon a\to b\),
respectively.

Define \(R\colon\mathcal D\to\mathcal D\) by
\[
R(p)=R(a)=p,
\qquad
R(b)=a,
\qquad
R(c)=c,
\]
and
\[
R(k)=1_p,
\qquad
R(h)=k,
\qquad
R(f_i)=f_ih
\quad (i=1,2,3).
\]
The components
\[
\varepsilon_p=1_p,
\qquad
\varepsilon_a=k,
\qquad
\varepsilon_b=h,
\qquad
\varepsilon_c=1_c
\]
define a right vector field
\[
\varepsilon\colon R\Rightarrow1_{\mathcal D}.
\]

Its fixed subcategory is
\[
\operatorname{Fix}(R)
=
\mathcal D|_{\{p,c\}}.
\]
Since \(R^2(a)=R^2(b)=p\), the endofunctor \(R^2\) takes values in
\(\operatorname{Fix}(R)\) and determines the associated right directed
deformation retract through
\[
\varepsilon^{(2)}
\colon
R^2\Rightarrow1_{\mathcal D}.
\]
Thus, the single vector field \(R\) encodes the successive down
beat-object reductions
\[
b\longmapsto a\longmapsto p,
\]
showing that stabilization may require more than one iteration.
\end{example}

\section{Homotopy theory and categorical flows}
\label{sec:homotopy_fluxes}

Natural transformations have long been interpreted as the building blocks for homotopies between functors \cite{LEEHomotopy}. Their intrinsic orientation leads to the directed homotopy theory studied by Grandis \cite{GrandisShape}. We recall the homotopy relations and lifting properties needed below. Finally, we introduce categorical flows as successive vector-field reductions on decreasing full subcategories.

\subsection{Strong and directed homotopies}
\label{subsec:strong_directed_homotopies}

Let \(F,G\colon\CC\to\Dc\) be functors. We write
\[
F\leq_dG
\]
if there is a natural transformation \(F\Rightarrow G\). This defines a
preorder compatible with composition of functors.

\begin{definition}\label{def:directed_homotopy}
A \emph{right directed homotopy} from \(F\) to \(G\) is a natural
transformation
\(
F\Rightarrow G,
\)
whereas a \emph{left directed homotopy} from \(F\) to \(G\) is a natural
transformation
\(
G\Rightarrow F.
\)
\end{definition}

Equivalently, a right directed homotopy is a functor
\(H\colon\CC\times\mathbb I_1\to\Dc\) whose restrictions to \(0\) and
\(1\) are \(F\) and \(G\), respectively.

\begin{definition}\label{def:strong_homotopy}
Two functors \(F,G\colon\CC\to\Dc\) are \emph{strongly homotopic},
written \(F\simeq_sG\), if they are connected by a finite zigzag of
natural transformations
\[
F=H_0
\leftrightsquigarrow
H_1
\leftrightsquigarrow
\cdots
\leftrightsquigarrow
H_n=G
\]
where \(\leftrightsquigarrow\) means either \(\Rightarrow\) or \(\Leftarrow\).
\end{definition}

\begin{remark}\label{rem:Homotopies_zig-zag}
By inserting identity transformations and composing consecutive
transformations with the same orientation, every strong homotopy may be
represented by an alternating zigzag
\[
H_0\Rightarrow H_1\Leftarrow H_2\Rightarrow\cdots
\leftrightsquigarrow H_n.
\]
\end{remark}

Thus, every directed homotopy is a strong homotopy, but a strong homotopy
need not have a uniform orientation.

We now recall connectedness and give an equivalent description of strong
homotopies.

For \(n\geq0\), let \(\mathbb I_n\) denote the category generated by the
alternating zigzag
\[
0
\longrightarrow
1
\longleftarrow
2
\longrightarrow
\cdots
\leftrightsquigarrow
n.
\]
Thus, for \(0\leq i<n\), the generating morphism is
\(i\to i+1\) when \(i\) is even and \(i+1\to i\) when \(i\) is odd.

A \emph{zigzag} from an object \(c\) to an object \(d\) of a category
\(\CC\) is a functor
\[
I\colon\mathbb I_n\longrightarrow\CC
\]
for some \(n\geq0\), such that \(I(0)=c\) and \(I(n)=d\). A category
\(\CC\) is \emph{connected} if it is non-empty and any two of its objects
are joined by a zigzag.

Being joined by a zigzag is an equivalence relation on
\(\operatorname{Ob}(\CC)\). Its equivalence classes determine the
\emph{connected components} of \(\CC\), and their set is denoted by
\(\pi_0(\CC)\); see
\cite[Chapter~IX, Section~3]{MacLaneCategories}.

\begin{proposition}
Two functors \(F,G\colon\CC\to\Dc\) are strongly homotopic if and only
if there exist \(m\geq0\) and a functor
\[
H\colon\CC\times\mathbb I_m\longrightarrow\Dc
\]
such that \(H(-,0)=F\) and \(H(-,m)=G\). Equivalently, \(F\) and \(G\)
belong to the same connected component of the functor category
\([\CC,\Dc]\).
\end{proposition}

\begin{proof}
    It is a consequence of Remark~\ref{rem:Homotopies_zig-zag}.
\end{proof}
With all of this we can define the notion of \emph{strong homotopy equivalence} and also of \emph{strong retract}. A functor \(F\colon\CC\to\Dc\) is a \emph{strong homotopy equivalence}
if there exists \(G\colon\Dc\to\CC\) such that
\[
GF\simeq_s1_\CC
\qquad\text{and}\qquad
FG\simeq_s1_\Dc.
\]
A full subcategory \(\Uc\subseteq\CC\) is a \emph{strong deformation
retract} of \(\CC\) if there are functors
\[
\iota\colon\Uc\hookrightarrow\CC,
\qquad
r\colon\CC\longrightarrow\Uc
\]
such that
\[
r\iota=1_\Uc
\qquad\text{and}\qquad
\iota r\simeq_s1_\CC.
\]
\subsection{Strong and directed fibrations}
\label{subsec:strong_directed_fibrations}

We recall the lifting properties introduced in
\cite{carcaciacampos2026weakstrongfibrationsfunctors,
DirectedHomotopyDatabases}.

\begin{definition}\label{def:right_directed_fibration_recall}
A functor \(P\colon\E\to\Ba\) is a \emph{right directed fibration} if,
for every natural transformation
\(
\alpha\colon F\Rightarrow G
\)
between functors \(F,G\colon\CC\to\Ba\), every lift
\(\widetilde F\colon\CC\to\E\) of \(F\) extends to a lift
\(\widetilde G\colon\CC\to\E\) of \(G\) and a natural transformation
\(
\widetilde\alpha\colon\widetilde F\Rightarrow\widetilde G
\)
such that \(P\widetilde\alpha=\alpha\).

A \emph{left directed fibration} is defined dually, by prescribing a lift
of \(G\).
\end{definition}

\begin{definition}\label{def:strong_fibration_recall}
A functor \(P\colon\E\to\Ba\) is a \emph{strong fibration} if every finite
zigzag of natural transformations in \(\Ba\) can be lifted after choosing
a lift of any one of its terms.
\end{definition}

\begin{proposition}\label{prop:bifibration_strong_fibration}
Every functor that is both a right and a left directed fibration is a
strong fibration.
\end{proposition}

\begin{proof}
Starting from a chosen lift, lift the transformations in the zigzag
successively, using the right or left lifting property according to their
orientation.
\end{proof}

\subsection{Vector fields and categorical flows}
\label{subsec:categorical_flows}
Flows and semiflows on finite \(T_0\)-spaces, equivalently finite
posets, have been studied by Chocano and collaborators from the
viewpoint of combinatorial dynamics. Every flow on an Alexandroff
\(T_0\)-space is trivial
\cite{ChocanoFlowsAlexandroff}, whereas non-trivial semiflows on a
finite \(T_0\)-space exist precisely in the presence of down beat
points and are closely related to strong deformation retracts
\cite{ChocanoSemiflows}.

Our use of the term \emph{flow} is discrete and categorical rather than
dynamical: it refers to a finite sequence of vector-field reductions,
each applied to the fixed subcategory obtained at the preceding stage.
The analogy lies in the fact that both constructions describe an
evolution towards smaller homotopy-equivalent models.

A vector field provides a one-step deformation of the identity. On a
finite acyclic category, its stabilization produces a directed
deformation retract onto the corresponding fixed subcategory. This
allows vector fields to be applied successively.
\begin{definition}\label{def:categorical_flow}
Let \(\CC\) be a finite acyclic category. A \emph{categorical flow} of
length \(n\) on \(\CC\) consists of a sequence
\[
\CC=\CC_0
\supseteq
\CC_1
\supseteq
\cdots
\supseteq
\CC_n
\]
of full subcategories and, for every \(1\leq k\leq n\), a right or left
vector field \(V_k\) on \(\CC_{k-1}\) such that
\[
\CC_k=\operatorname{Fix}(V_k).
\]

The category \(\CC_n\) is called the \emph{terminal category} of the
flow. The flow is \emph{right-directed} if every \(V_k\) is a right
vector field, \emph{left-directed} if every \(V_k\) is a left vector
field, and \emph{mixed} otherwise.
\end{definition}

\begin{proposition}\label{prop:flows_preserve_strong_homotopy}
Let
\[
\CC=\CC_0
\supseteq
\CC_1
\supseteq
\cdots
\supseteq
\CC_n
\]
be a categorical flow. Then, for every \(1\leq k\leq n\), the inclusion
\(
\iota_k\colon\CC_k\hookrightarrow\CC_{k-1}
\)
is a strong homotopy equivalence. Consequently, the composite inclusion
\(
\iota
=
\iota_1\cdots\iota_n
\colon
\CC_n\hookrightarrow\CC
\)
is a strong homotopy equivalence, and hence
\(
\CC_n\simeq_s\CC.
\)
\end{proposition}

\begin{proof}
Fix \(1\leq k\leq n\). If \(V_k\) is a right vector field, then
Theorem~\ref{thm:vector_field_stabilization} yields a retraction
\(
r_k\colon\CC_{k-1}\longrightarrow\CC_k
\)
such that
\(
r_k\iota_k=1_{\CC_k}
\)
and a natural transformation
\(
\iota_kr_k\Rightarrow1_{\CC_{k-1}}.
\)
Thus, \(\iota_k\) and \(r_k\) are mutually inverse up to strong
homotopy.

If \(V_k\) is a left vector field, the dual stabilization theorem gives
the same identities together with a natural transformation
\[
1_{\CC_{k-1}}\Rightarrow\iota_kr_k,
\]
which again makes \(\iota_k\) a strong homotopy equivalence.

Since strong homotopy equivalences are closed under composition, the
composite inclusion
\[
\CC_n\hookrightarrow\CC_{n-1}
\hookrightarrow\cdots\hookrightarrow\CC_0=\CC
\]
is a strong homotopy equivalence.
\end{proof}

\section{Sectional category under vector-field reductions}
\label{sec:sectional_category_reductions}

We briefly recall the directed sectional category introduced in
\cite{DirectedHomotopyDatabases}. A \emph{cover} of a small category
\(\CC\) is a family of subcategories
\[
\{\Uc_0,\ldots,\Uc_n\}
\]
such that every object and every morphism of \(\CC\) belongs to at least
one \(\Uc_i\).

Let \(P\colon\E\to\CC\) be a functor, and let
\(\iota_i\colon\Uc_i\hookrightarrow\CC\) denote the inclusion. A
\emph{local right homotopy section} of \(P\) over \(\Uc_i\) consists of a
functor
\(
s_i\colon\Uc_i\longrightarrow\E
\)
and a natural transformation
\(
\alpha_i\colon Ps_i\Rightarrow\iota_i.
\)
We recall the definition of directed sectional category given by the author in \cite[Definition 6.2]{DirectedHomotopyDatabases}.
\begin{definition}
The \emph{directed sectional category} of \(P\), denoted by
\( \allowbreak\mathrm{dsecat}(P)\), is the least integer \(n\geq0\) for which
\(\CC\) admits a cover
\(
\{\Uc_0,\ldots,\Uc_n\}
\)
such that \(P\) has a local right homotopy section over every \(\Uc_i\).
If no such finite cover exists, we set
\(
\mathrm{dsecat}(P)=\infty.
\)
\end{definition}

If instead of local right homotopy sections we use strict sections, i.e. \(P s_i=\iota_i\), we have the \emph{sectional category} \cite{Baues-Isaac}.

We shall also use the following change-of-base notation. Given a functor
\[
f\colon\Dc\longrightarrow\CC,
\]
let
\[
f^*\E=\Dc\times_\CC\E
\]
and denote by
\[
f^*P\colon f^*\E\longrightarrow\Dc
\]
the pullback of \(P\) along \(f\):
\[
\begin{tikzcd}
f^*\E
  \arrow[r]
  \arrow[d, "f^*P"']
&
\E
  \arrow[d, "P"]
\\
\Dc
  \arrow[r, "f"']
&
\CC .
\end{tikzcd}
\]
Thus, \(f^*P\) is the change of base of \(P\) along \(f\).

We first establish two lemmas relating directed sectional category to right directed fibrations.

\begin{lemma}\label{lem:right_fibration_strictification}
Let \(P\colon\E\to\CC\) be a right directed fibration. Every local right
homotopy section of \(P\) can be replaced by a strict local section.
Consequently,
\[
\mathrm{dsecat}(P)=\mathrm{secat}(P).
\]
\end{lemma}

\begin{proof}
Let \(s\colon\Uc\to\E\) and
\[
\alpha\colon Ps\Rightarrow\iota
\]
be a local right homotopy section, where
\(\iota\colon\Uc\hookrightarrow\CC\) is the inclusion. Since \(P\) is a
right directed fibration, \(\alpha\) lifts from \(s\). Hence there are a
functor \(\widetilde s\colon\Uc\to\E\) and a natural transformation
\(
s\Rightarrow\widetilde s
\)
such that
\(
P\widetilde s=\iota.
\)
Thus, \(\widetilde s\) is a strict local section. The reverse inequality
\(\mathrm{dsecat}(P)\leq\mathrm{secat}(P)\) follows because every strict
section is a right homotopy section via the identity transformation.
\end{proof}
\begin{lemma}\label{lem:dsecat_pullback_monotonicity}
Let \(P\colon\E\to\CC\) be a right directed fibration and let
\(f\colon\Dc\to\CC\) be a functor. Then
\[
\mathrm{dsecat}(f^*P)
\leq
\mathrm{dsecat}(P).
\]
\end{lemma}

\begin{proof}
Let \(\{\Uc_0,\ldots,\Uc_n\}\) be a cover of \(\CC\) admitting local
right homotopy sections of \(P\). By
Lemma~\ref{lem:right_fibration_strictification}, these sections may be
chosen strict.

The inverse-image subcategories
\[
f^{-1}(\Uc_0),\ldots,f^{-1}(\Uc_n)
\]
form a cover of \(\Dc\), since covers are preserved by inverse image.
Each strict local section of \(P\) over \(\Uc_i\) pulls back to a strict
local section of \(f^*P\) over \(f^{-1}(\Uc_i)\). Therefore,
\[
\mathrm{dsecat}(f^*P)
\leq
n,
\]
and the result follows.
\end{proof}
We now show that a right directed deformation of the base preserves the
directed sectional category of right directed fibrations. The result will
then be applied to the retract obtained by stabilizing a right vector
field.

\begin{theorem}\label{thm:dsecat_directed_fibration_retract}
Let
\[
\iota\colon\Uc\longrightarrow\CC,
\qquad
r\colon\CC\longrightarrow\Uc
\]
be functors, and suppose that there is a natural transformation
\[
\alpha\colon\iota r\Rightarrow1_\CC.
\]
If \(P\colon\E\to\CC\) is a right directed fibration, then
\[
\mathrm{dsecat}(\iota^*P)
=
\mathrm{dsecat}(P).
\]
\end{theorem}

\begin{proof}
Since \(\iota^*P\) is obtained from \(P\) by change of base along
\(\iota\), using Lemma~\ref{lem:dsecat_pullback_monotonicity} we obtain
\[
\mathrm{dsecat}(\iota^*P)
\leq
\mathrm{dsecat}(P).
\]

Consider now the two successive pullbacks
\[
\begin{tikzcd}[column sep=large]
r^*(\iota^*\E)
  \arrow[r]
  \arrow[d, "Q"']
&
\iota^*\E
  \arrow[r]
  \arrow[d, "\iota^*P"']
&
\E
  \arrow[d, "P"]
\\
\CC
  \arrow[r, "r"']
&
\Uc
  \arrow[r, "\iota"']
&
\CC .
\end{tikzcd}
\]
Both squares are pullbacks, so the outer rectangle is also a pullback.
Consequently, there is a canonical isomorphism
\[
Q
=
r^*(\iota^*P)
\cong
(\iota r)^*P.
\]

Let
\[
J\colon r^*(\iota^*\E)\longrightarrow\E
\]
be the canonical functor determined by the outer pullback. Then
\[
PJ=\iota rQ.
\]
Whiskering \(\alpha\colon\iota r\Rightarrow1_\CC\) with \(Q\) gives a
natural transformation
\[
\alpha Q\colon
PJ=\iota rQ
\Rightarrow
Q.
\]

Since \(P\) is a right directed fibration and \(J\) is a lift of
\(\iota rQ\), this transformation admits a lift starting at \(J\).
Hence there are a functor
\[
A_\alpha\colon r^*(\iota^*\E)\longrightarrow\E
\]
and a natural transformation
\(
J\Rightarrow A_\alpha
\)
such that
\(
PA_\alpha=Q.
\)
Therefore, by \cite[Proposition~6.3]{DirectedHomotopyDatabases},
\[
\mathrm{dsecat}(P)
\leq
\mathrm{dsecat}(PA_\alpha)
=
\mathrm{dsecat}(Q).
\]

Finally, since
\(
Q=r^*(\iota^*P),
\)
Lemma~\ref{lem:dsecat_pullback_monotonicity}, applied along \(r\), gives
\[
\mathrm{dsecat}(Q)
\leq
\mathrm{dsecat}(\iota^*P).
\]
Combining the three inequalities, we obtain
\[
\mathrm{dsecat}(P)
\leq
\mathrm{dsecat}(Q)
\leq
\mathrm{dsecat}(\iota^*P)
\leq
\mathrm{dsecat}(P),
\]
and hence all three quantities are equal.
\end{proof}

\begin{remark}
    Notice that the theorem does not require \(r\iota=1_\Uc\); only the
directed deformation \(\iota r\Rightarrow1_\CC\) is used.
\end{remark}
\begin{corollary}\label{cor:dsecat_vector_field_reduction}
Let \(\CC\) be a finite acyclic category and let
\((R,\varepsilon)\) be a right vector field on \(\CC\). If
\[
\iota_R\colon\operatorname{Fix}(R)\hookrightarrow\CC
\]
is the inclusion, then every right directed fibration
\(P\colon\E\to\CC\) satisfies
\[
\mathrm{dsecat}(P)
=
\mathrm{dsecat}(\iota_R^*P).
\]
\end{corollary}

\begin{proof}
By Theorem~\ref{thm:vector_field_stabilization}, a sufficiently large
iterate of \(R\) corestricts to a functor
\[
r_R\colon\CC\longrightarrow\operatorname{Fix}(R),
\]
and the iterated transformation yields
\[
\iota_Rr_R\Rightarrow1_\CC.
\]
The result follows from
Theorem~\ref{thm:dsecat_directed_fibration_retract}.
\end{proof}

Thus, the directed sectional category of a right directed fibration may be
computed after restricting its base to the fixed subcategory of a right
vector field.

\section{Vector fields and databases}
\label{sec:databases_retractions}

\subsection{Functorial databases and global sections}
\label{subsec:functorial_databases_sections}

In the functorial model of databases, a schema is a small category
\(\CC\), and a database instance on it is a functor
\(
X\colon\CC\longrightarrow\Sets;
\)
see \cite{SpivakDataMigration,SpivakDatabaseQueries,SevenSketches}.
The elements of \(X(c)\) are the records of type \(c\), while the
morphisms of \(\CC\) encode the functional relationships among data
types.

Equivalently, \(X\) determines a discrete opfibration
\[
\pi_X\colon\int^\CC X\longrightarrow\CC.
\]
A section of \(\pi_X\) is a globally coherent choice of one record of
each type.

\begin{definition}\label{def:global_sections_database}
The set of \emph{global sections} of a database
\(X\colon\CC\to\Sets\) is
\[
\Gamma(\CC,X)
=
\operatorname{Sect}(\pi_X).
\]
\end{definition}

There is a natural identification
\[
\Gamma(\CC,X)\cong\lim_\CC X.
\]
Indeed, a section is a family \((x_c)_{c\in\CC}\), with
\(x_c\in X(c)\), such that
\(
X(f)(x_c)=x_d
\)
for every morphism \(f\colon c\to d\).

We now show that right directed deformation retracts preserve all limits,
and hence all global sections.

\begin{proposition}\label{prop:directed_retract_initial}
Let
\[
\Uc
\mathrel{\substack{\xrightarrow{\iota}\\[-1mm]\xleftarrow[r]{}}}
\CC
\]
be a right directed deformation retract. Then the inclusion
\(
\iota\colon\Uc\hookrightarrow\CC
\)
is initial.
\end{proposition}

\begin{proof}
Let
\(
\alpha\colon\iota r\Rightarrow1_\CC
\)
be the associated natural transformation. Thus,
\(
r\iota=1_\Uc
\text{ and }
\alpha\iota=1_\iota.
\)

For every \(c\in\CC\), the pair
\[
\bigl(r(c),\alpha_c\colon\iota r(c)\to c\bigr)
\]
is an object of the comma category \(\iota\downarrow c\). Hence
\(\iota\downarrow c\) is non-empty.

Let
\[
(u,f\colon\iota(u)\to c)
\]
be any other object of \(\iota\downarrow c\). Naturality of \(\alpha\)
with respect to \(f\) gives the commutative square
\[
\begin{tikzcd}[column sep=large]
\iota r\iota(u)
  \arrow[r, "\iota r(f)"]
  \arrow[d, "\alpha_{\iota(u)}"']
&
\iota r(c)
  \arrow[d, "\alpha_c"]
\\
\iota(u)
  \arrow[r, "f"']
&
c.
\end{tikzcd}
\]
Since \(r\iota=1_\Uc\) and
\(\alpha_{\iota(u)}=1_{\iota(u)}\), this square yields
\[
f=\alpha_c\,\iota(r(f)).
\]
Therefore,
\[
r(f)\colon u\longrightarrow r(c)
\]
defines a morphism
\[
r(f)\colon
(u,f)
\longrightarrow
(r(c),\alpha_c)
\]
in \(\iota\downarrow c\).

Thus, every object of \(\iota\downarrow c\) is connected to
\((r(c),\alpha_c)\). Consequently, \(\iota\downarrow c\) is non-empty
and connected for every \(c\in\CC\), and hence \(\iota\) is initial.
\end{proof}

\begin{corollary}\label{cor:directed_retract_limits}
Under the hypotheses of
Proposition~\ref{prop:directed_retract_initial}, restriction along
\(\iota\) induces a natural isomorphism
\[
\lim_\CC X
\cong
\lim_\Uc X\iota
\]
for every functor \(X\colon\CC\to\E\) for which these limits exist.
\end{corollary}

\begin{proof}
This is the cofinality theorem for initial functors \cite[Chapter~IX, Section~3]{MacLaneCategories}.
\end{proof}

\begin{corollary}\label{cor:directed_retract_database_sections}
Under the same hypotheses, restriction induces a natural bijection
\[
\Gamma(\CC,X)
\cong
\Gamma(\Uc,X|_\Uc)
\]
for every database \(X\colon\CC\to\Sets\).
\end{corollary}

\begin{proof}
This follows from
Corollary~\ref{cor:directed_retract_limits} and the identification of
global sections with limits.
\end{proof}

\begin{corollary}\label{cor:vector_field_database_sections}
Let \(\CC\) be a finite acyclic category and let
\((R,\varepsilon)\) be a right vector field on \(\CC\). Then restriction
induces a natural bijection
\[
\Gamma(\CC,X)
\cong
\Gamma\bigl(
\operatorname{Fix}(R),
X|_{\operatorname{Fix}(R)}
\bigr)
\]
for every database \(X\colon\CC\to\Sets\).
\end{corollary}

\begin{proof}
By Theorem~\ref{thm:vector_field_stabilization}, a sufficiently large
iterate of \(R\) determines a right directed deformation retract
\[
\operatorname{Fix}(R)
\mathrel{\substack{\xrightarrow{\iota_R}\\[-1mm]\xleftarrow[r_R]{}}}
\CC
\]
relative to \(\operatorname{Fix}(R)\). The result follows from
Corollary~\ref{cor:directed_retract_database_sections}.
\end{proof}

\subsection{Preservation of sectional category}
\label{subsec:sectional_category_reduction}

For Grothendieck opfibrations, the preceding reduction also preserves
ordinary sectional category.

\begin{corollary}\label{cor:secat_opfibration_vector_field}
Let \(\CC\) be a finite acyclic category, let
\((R,\varepsilon)\) be a right vector field, and let
\[
\iota_R\colon\operatorname{Fix}(R)\hookrightarrow\CC
\]
be the inclusion. If \(P\colon\E\to\CC\) is a Grothendieck opfibration,
then
\[
\mathrm{secat}(\pi_X)
=
\mathrm{dsecat}(\pi_X)
=
\mathrm{dsecat}
\bigl(\pi_{X|_{\operatorname{Fix}(R)}}\bigr)
=
\mathrm{secat}
\bigl(\pi_{X|_{\operatorname{Fix}(R)}}\bigr).
\]
\end{corollary}

\begin{proof}
Every Grothendieck opfibration is a right directed fibration, and every
local right homotopy section of such an opfibration can be strictified
\cite[Propositions~2.6 and~6.4]{DirectedHomotopyDatabases}. Therefore,
Corollary~\ref{cor:dsecat_vector_field_reduction} gives
\[
\mathrm{secat}(\pi_X)
=
\mathrm{dsecat}(\pi_X)
=
\mathrm{dsecat}
\bigl(\pi_{X|_{\operatorname{Fix}(R)}}\bigr)
=
\mathrm{secat}
\bigl(\pi_{X|_{\operatorname{Fix}(R)}}\bigr).
\]
\end{proof}

\begin{corollary}\label{cor:database_vector_field_reduction}
Let \(\CC\) be a finite acyclic category, let
\((R,\varepsilon)\) be a right vector field, and let
\(X\colon\CC\to\Sets\) be a database. Then
\[
\mathrm{dsecat}(\pi_X)
=
\mathrm{dsecat}
\bigl(\pi_{X|_{\operatorname{Fix}(R)}}\bigr),
\]
and restriction induces a natural bijection
\[
\Gamma(\CC,X)
\cong
\Gamma\bigl(
\operatorname{Fix}(R),
X|_{\operatorname{Fix}(R)}
\bigr).
\]
\end{corollary}
\begin{proof}
The category-of-elements projection \(\pi_X\) is a discrete
Grothendieck opfibration, and there is a canonical isomorphism
\[
\iota_R^*\pi_X
\cong
\pi_{X|_{\operatorname{Fix}(R)}}.
\]
The equality of sectional categories follows from
Corollary~\ref{cor:secat_opfibration_vector_field}, while the bijection
of global sections follows from
Corollary~\ref{cor:vector_field_database_sections}.
\end{proof}

Thus, right vector-field reductions preserve both the global sections of
a functorial database and the sectional category of its
category-of-elements projection.

\section{Initial scaffolds of acyclic categories}
\label{sec:initial-scaffolds}

We now introduce a second reduction procedure for acyclic categories.
Unlike vector-field reductions, which produce directed deformation
retracts, initial scaffolds produce initial subcategories and therefore
preserve all limits. Dey and Lesnick introduced this construction for
posets with finite downsets as an explicit model for a minimal initial
functor
\cite{dey2026limitcomputationposetsminimal}. We extend it to lower
well-founded acyclic categories by replacing open downsets with punctured
lower comma categories, with particular emphasis on set-valued functors
arising as database instances.

Throughout this section, \(\CC\) is an acyclic category. The existence of
a non-identity morphism defines a strict partial order on its objects:
\[
x<_\CC y
\quad\Longleftrightarrow\quad
\Hom_\CC(x,y)\neq\varnothing.
\]
We say that \(\CC\) is \emph{lower well-founded} if this order is
well-founded, equivalently, if there is no infinite sequence of
non-identity morphisms
\[
\cdots\longrightarrow x_2\longrightarrow x_1\longrightarrow x_0.
\]
Every finite acyclic category is lower well-founded.

For \(c\in\CC\), recall that
\[
\partial(\CC\downarrow c)
=
(\CC\downarrow c)\setminus\{(c,1_c)\}
\]
is the punctured lower comma category of \(c\). When \(\CC\) is a poset,
it is naturally identified with the elements below \(c\).

An object \(m\in\CC\) is a \emph{source} if it receives no non-identity
morphism, equivalently, if
\[
\partial(\CC\downarrow m)=\varnothing.
\]
Set
\[
M_\CC
=
\{m\in\operatorname{Ob}(\CC)
\mid
\partial(\CC\downarrow m)=\varnothing\}
\]
and
\[
E_\CC
=
\left\{
c\in\operatorname{Ob}(\CC)
\ \middle|\
\partial(\CC\downarrow c)\neq\varnothing
\text{ and }
\partial(\CC\downarrow c)\text{ is disconnected}
\right\}.
\]
Finally, put
\[
I_\CC=M_\CC\cup E_\CC.
\]

Separating sources from lower essential objects allows us to avoid the convention, adopted in \cite{dey2026limitcomputationposetsminimal}, that the empty category is disconnected.

\begin{lemma}\label{lem:source-in-each-component}
Let \(\CC\) be lower well-founded. Every connected component of
\(\partial(\CC\downarrow c)\) contains an object
\[
(m,\alpha\colon m\to c)
\]
whose domain \(m\) is a source.
\end{lemma}

\begin{proof}
Starting with an object \((x,\beta\colon x\to c)\), choose a
non-identity morphism \(x_1\to x\) whenever \(x\) is not a source, and
continue recursively. Lower well-foundedness implies that this process
terminates at a source \(m\). The resulting composite \(v\colon m\to x\)
defines a morphism
\[
v\colon(m,\beta v)\longrightarrow(x,\beta)
\]
in \(\partial(\CC\downarrow c)\), so both objects lie in the same connected component.
\end{proof}

\subsection{Initial scaffolds}

\begin{definition}\label{def:categorical-initial-scaffold}
Let \(\CC\) be a lower well-founded acyclic category. An
\emph{initial scaffold} of \(\CC\) is a subcategory
\(\mathcal P\subseteq\CC\) constructed as follows:

\begin{enumerate}
    \item
    \[
    \operatorname{Ob}(\mathcal P)=I_\CC.
    \]

    \item For every \(c\in E_\CC\) and every connected component
    \(A\in\pi_0(\partial(\CC\downarrow c))\), choose an object
    \[
    (m_A,\alpha_A\colon m_A\to c)\in A
    \]
    with \(m_A\in M_\CC\).

    \item The non-identity morphisms of \(\mathcal P\) are precisely the
    chosen morphisms \(\alpha_A\colon m_A\to c\).
\end{enumerate}
\end{definition}

The choices in the definition exist by
Lemma~\ref{lem:source-in-each-component}. They determine a subcategory
because the domain of every chosen non-identity morphism is a source,
whereas its codomain lies in \(E_\CC\) and is therefore not a source.
Hence no two chosen non-identity morphisms are composable.

For posets, this recovers the usual initial scaffold: its objects are the
minimal and lower essential elements, and one relation from a minimal
element is chosen in each connected component of the open lower set of
every lower essential element.

\subsection{Initiality}

Let \(j\colon\mathcal P\hookrightarrow\CC\) be the inclusion. Every
morphism \(\alpha\colon x\to c\) induces by postcomposition a functor
\[
\alpha_*
\colon
(j\downarrow x)\longrightarrow(j\downarrow c),
\qquad
(p,\gamma)\longmapsto(p,\alpha\gamma).
\]

\begin{lemma}\label{lem:connected-lower-comma}
Let \(c\notin\operatorname{Ob}(\mathcal P)\), and suppose that
\((j\downarrow x)\) is non-empty and connected for every \(x<_\CC c\).
If \(\partial(\CC\downarrow c)\) is non-empty and connected, then
\((j\downarrow c)\) is non-empty and connected.
\end{lemma}

\begin{proof}
For every
\[
a=(x,\alpha)\in\partial(\CC\downarrow c),
\]
let \(S_a\) be the full subcategory of \(j\downarrow c\) spanned by the
objects in the image of
\[
\alpha_*
\colon
(j\downarrow x)\longrightarrow(j\downarrow c).
\]
Since \(x<_\CC c\), the category \(j\downarrow x\) is non-empty and
connected by hypothesis. Hence \(S_a\) is non-empty and connected.
A morphism
\[
u\colon(x,\alpha)\longrightarrow(y,\beta)
\]
in \(\partial(\CC\downarrow c)\) satisfies
\(\alpha=\beta u\), and hence
\(
\alpha_*=\beta_*u_*.
\)
Therefore,
\(
S_{(x,\alpha)}
\subseteq
S_{(y,\beta)}.
\)

Since \(\partial(\CC\downarrow c)\) is connected, any two of its objects
are joined by a zigzag. The corresponding subcategories \(S_a\) form a
zigzag of inclusions, so consecutive members have non-empty intersection.
As every \(S_a\) is non-empty and connected, their union is connected.

Finally, every \((p,\gamma)\in(j\downarrow c)\) belongs to this union.
Indeed, \(c\notin\operatorname{Ob}(\mathcal P)\) implies that \(\gamma\)
is non-identity, so \((p,\gamma)\in \partial(\CC\downarrow c)\), and
\[
(p,\gamma)=\gamma_*(p,1_p)\in S_{(p,\gamma)}.
\]
Thus, the \(S_a\) cover \((j\downarrow c)\), which is therefore
non-empty and connected.
\end{proof}

\begin{theorem}\label{thm:initial-scaffold-is-initial}
Let \(\CC\) be a lower well-founded acyclic category and let
\(\mathcal P\subseteq\CC\) be an initial scaffold. Then the inclusion
\[
j\colon\mathcal P\hookrightarrow\CC
\]
is initial.
\end{theorem}

\begin{proof}
We prove by well-founded induction on \(<_\CC\) that
\((j\downarrow c)\) is non-empty and connected for every \(c\in\CC\).

For the base case, a minimal element \(c\) for the order \(<_\CC\)\ is an element in \(M_\CC\), hence \(c\in\mathcal P\), and the only morphism with
codomain \(c\) is \(1_c\). Hence
\(
(j\downarrow c)=\{(c,1_c)\}.
\)

Now assume that the statement is true for all \(x <_\CC c\). If \(c\notin I_\CC\), then \(\partial(\CC\downarrow c)\) is non-empty and connected. The
induction hypothesis applies to the domain of every object of
\(\partial(\CC\downarrow c)\), so the conclusion follows from
Lemma~\ref{lem:connected-lower-comma}.

It remains to consider the case \(c\in E_\CC\). Since
\(c\in\operatorname{Ob}(\mathcal P)\), the pair
\(
z_c=(c,1_c)
\)
is an object of \(j\downarrow c\).

Fix a connected component
\[
A\in\pi_0\bigl(\partial(\CC\downarrow c)\bigr).
\]
For each \(a=(x,\alpha)\in A\), let \(S_a\) be the full subcategory of
\(j\downarrow c\) spanned by the objects in the image of
\[
\alpha_*
\colon
(j\downarrow x)\longrightarrow(j\downarrow c).
\]
Since \(x<_\CC c\), the induction hypothesis implies that
\(j\downarrow x\) is non-empty and connected. Therefore, \(S_a\) is also
non-empty and connected.

If
\[
u\colon(x,\alpha)\longrightarrow(y,\beta)
\]
is a morphism in \(A\), then \(\alpha=\beta u\), and hence
\(
\alpha_*=\beta_*u_*.
\)
It follows that
\(
S_{(x,\alpha)}
\subseteq
S_{(y,\beta)}.
\)

Since \(A\) is connected, any two of its objects are joined by a zigzag.
The corresponding subcategories \(S_a\) form a zigzag of inclusions, so
consecutive members have non-empty intersection.

Let \(U_A\) be the full subcategory of \(j\downarrow c\) spanned by the
objects belonging to some \(S_a\), with \(a\in A\). The preceding
argument shows that \(U_A\) is non-empty and connected.

Now let
\[
(m_A,\alpha_A\colon m_A\to c)\in A
\]
be the representative chosen in the scaffold. Since \(m_A\in M_\CC\),
we have \(m_A\in\operatorname{Ob}(\mathcal P)\), and
\[
(m_A,\alpha_A)
=
(\alpha_A)_*(m_A,1_{m_A})
\]
is an object of \(U_A\). Moreover, since \(\alpha_A\) is a morphism of
\(\mathcal P\), it determines a morphism
\[
(m_A,\alpha_A)\longrightarrow(c,1_c)=z_c
\]
in \(j\downarrow c\). Hence every object of \(U_A\) is connected to
\(z_c\).

Finally, let \((p,\gamma)\) be any object of \(j\downarrow c\) distinct
from \(z_c\). Then \(\gamma\colon p\to c\) is non-identity and therefore
determines an object
\[
(p,\gamma)\in\partial(\CC\downarrow c).
\]
Let \(A\) be its connected component. Since
\[
(p,\gamma)=\gamma_*(p,1_p),
\]
the object \((p,\gamma)\) belongs to \(U_A\), and is consequently
connected to \(z_c\).

Thus, every object of \(j\downarrow c\) is connected to \(z_c\). Since
\(z_c\) is itself an object of \(j\downarrow c\), this comma category is
non-empty and connected. The induction is complete, and therefore \(j\)
is initial.
\end{proof}
\begin{remark}\label{rem:finite-case}
For a finite acyclic category, the preceding argument can be carried out
by ordinary induction along any linear extension of \(<_\CC\).
\end{remark}

\begin{corollary}\label{cor:scaffold-preserves-limits}
Let \(\mathcal P\subseteq\CC\) be an initial scaffold. For every functor
\(F\colon\CC\to\E\) for which the relevant limits exist, restriction
induces a natural isomorphism
\[
\lim_\CC F
\cong
\lim_{\mathcal P}F|_{\mathcal P}.
\]
\end{corollary}

\begin{proof}
As in Corollary~\ref{cor:directed_retract_limits}, this is the dual of the finality theorem for final functors \cite[Chapter~IX, Section~3]{MacLaneCategories}.
\end{proof}

\begin{corollary}\label{cor:scaffold-database-sections}
For every database \(X\colon\CC\to\Sets\), restriction to an initial
scaffold induces a natural bijection
\[
\Gamma(\CC,X)
\cong
\Gamma\bigl(\mathcal P,X|_{\mathcal P}\bigr).
\]
\end{corollary}

\begin{proof}
This follows from
Corollary~\ref{cor:scaffold-preserves-limits} and the identification of
global sections with limits.
\end{proof}

Thus, an initial scaffold may be regarded as a reduced database schema
that preserves, uniformly for every instance, its globally coherent
selections. Unlike a vector-field reduction, it need not be a directed
deformation retract; its defining property is the preservation of all
limits.

\subsection{Invariance under vector-field reductions}

We now relate initial scaffolds to the directed reductions studied in
Section~\ref{sec:homotopy_retracts_vector_fields}. The construction of an
initial scaffold depends on the sources, the lower essential objects, and
the connected components of the punctured lower comma categories. We shall
show that all these data are preserved by right directed deformation
retracts.

We first record that a directed retract induces a bijection on connected
components.

\begin{lemma}\label{lem:directed_retract_components}
Let
\[
\CC
\mathrel{\substack{\xrightarrow{F}\\[-1mm]\xleftarrow[G]{}}}
\Dc
\]
be functors such that \(GF=1_\CC\), and suppose that there is a natural
transformation
\(
FG\Rightarrow1_\Dc
\)
or
\(
1_\Dc\Rightarrow FG.
\)
Then \(F\) induces a bijection
\[
\pi_0(F)\colon\pi_0(\CC)\longrightarrow\pi_0(\Dc),
\]
whose inverse is induced by \(G\).
\end{lemma}

\begin{proof}
Every functor sends zigzags to zigzags and therefore induces a map on
connected components. Moreover, naturally related functors induce the
same map on connected components: if
\(\eta\colon H\Rightarrow K\) or \(\eta\colon K\Rightarrow H\), then
\(\eta_c\) joins \(H(c)\) and \(K(c)\) for every object \(c\).

The equality \(GF=1_\CC\) gives
\[
\pi_0(G)\pi_0(F)
=
1_{\pi_0(\CC)}.
\]
Similarly, the natural transformation between \(FG\) and \(1_\Dc\)
implies
\[
\pi_0(F)\pi_0(G)
=
\pi_0(FG)
=
1_{\pi_0(\Dc)}.
\]
Thus, \(\pi_0(F)\) and \(\pi_0(G)\) are mutually inverse.
\end{proof}

We can now prove the invariance result at the level of directed
deformation retracts.

\begin{proposition}\label{prop:right_retract_preserves_scaffolds}
Let \(\CC\) be a lower well-founded acyclic category, and let
\[
\Uc
\mathrel{\substack{\xrightarrow{\iota}\\[-1mm]\xleftarrow[r]{}}}
\CC
\]
be a right directed deformation retract onto a full subcategory. Then
\(\CC\) and \(\Uc\) have the same initial scaffolds. More precisely, a
subcategory \(\mathcal{P}\subseteq\Uc\) is an initial scaffold of \(\Uc\) if and
only if it is an initial scaffold of \(\CC\).
\end{proposition}

\begin{proof}
Let
\(
\alpha\colon\iota r\Rightarrow1_\CC
\)
be the natural transformation associated with the retract. Thus,
\(
r\iota=1_\Uc
\text{ and }
\alpha\iota=1_\iota.
\)

We first compare the punctured lower comma categories at the objects of
\(\Uc\). For every \(d\in\Uc\), the inclusion induces a functor
\[
\iota_d\colon
\partial(\Uc\downarrow d)
\longrightarrow
\partial(\CC\downarrow d).
\]
Define
\[
r_d\colon
\partial(\CC\downarrow d)
\longrightarrow
\partial(\Uc\downarrow d)
\]
on objects by
\[
r_d(x,f)
=
\bigl(r(x),f\alpha_x\bigr)
\]
and on morphisms by applying \(r\).

Let us check that \(r_d\) takes values in the punctured comma category.
If \(f\alpha_x=1_d\), then \(r(x)=d\), and hence
\[
d=\iota r(x)
\xrightarrow{\alpha_x}
x
\xrightarrow{f}
d
\]
is a directed cycle. Since \(\CC\) is acyclic, both morphisms must be
identities. It follows that \(x=d\) and \(f=1_d\), contradicting
\[
(x,f)\in\partial(\CC\downarrow d).
\]
Thus \(f\alpha_x\) is non-identity.

The functors \(\iota_d\) and \(r_d\) satisfy
\(
r_d\iota_d
=
1_{\partial(\Uc\downarrow d)}.
\)
Moreover, the components of \(\alpha\) define a natural transformation
\(
\iota_dr_d
\Rightarrow
1_{\partial(\CC\downarrow d)}.
\)

Indeed, for every \((x,f)\in\partial(\CC\downarrow d)\), the component is
the morphism
\[
\alpha_x\colon
\bigl(\iota r(x),f\alpha_x\bigr)
\longrightarrow
(x,f).
\]
Therefore, Lemma~\ref{lem:directed_retract_components} gives a bijection
\[
\pi_0\bigl(\partial(\Uc\downarrow d)\bigr)
\cong
\pi_0\bigl(\partial(\CC\downarrow d)\bigr).
\]

We next show that no object of \(\CC\setminus\Uc\) belongs to \(I_\CC\).
Let \(c\in\CC\setminus\Uc\). The component
\(
\alpha_c\colon\iota r(c)\longrightarrow c
\)
is non-identity, since its domain belongs to \(\Uc\) whereas \(c\) does
not. Hence \(\partial(\CC\downarrow c)\) is non-empty.

Let
\(
(x,f\colon x\to c)
\)
be any object of \(\partial(\CC\downarrow c)\). Naturality of \(\alpha\)
gives the commutative square
\[
\begin{tikzcd}[column sep=large]
\iota r(x)
  \arrow[r, "\iota r(f)"]
  \arrow[d, "\alpha_x"']
&
\iota r(c)
  \arrow[d, "\alpha_c"]
\\
x
  \arrow[r, "f"']
&
c.
\end{tikzcd}
\]
Thus, in \(\partial(\CC\downarrow c)\), we have a zigzag
\[
\begin{tikzcd}[column sep=large]
&
\bigl(\iota r(x),f\alpha_x\bigr)
  \arrow[dl, "\alpha_x"']
  \arrow[dr, "\iota r(f)"]
&
\\
(x,f)
&
&
\bigl(\iota r(c),\alpha_c\bigr).
\end{tikzcd}
\]
Consequently, every object of \(\partial(\CC\downarrow c)\) is connected
to \((\iota r(c),\alpha_c)\). Hence this punctured comma category is
non-empty and connected, and therefore
\(
c\notin I_\CC.
\)

We have shown that every object of \(I_\CC\) belongs to \(\Uc\). For an
object \(d\in\Uc\), the preceding comparison of punctured comma
categories shows that \(d\) is a source in \(\Uc\) if and only if it is a
source in \(\CC\), and that \(d\) is lower essential in \(\Uc\) if and
only if it is lower essential in \(\CC\). Therefore,
\[
M_\Uc=M_\CC,
\qquad
E_\Uc=E_\CC,
\qquad
I_\Uc=I_\CC,
\]
where these sets are identified through the inclusion
\(\operatorname{Ob}(\Uc)\subseteq\operatorname{Ob}(\CC)\).

Moreover, for every \(d\in E_\Uc=E_\CC\), the functor \(\iota_d\)
identifies the connected components of
\(\partial(\Uc\downarrow d)\) with those of
\(\partial(\CC\downarrow d)\). The source representatives used in the
definition of a scaffold also belong to \(\Uc\): if
\(m\in M_\CC\), then
\[
\alpha_m\colon\iota r(m)\longrightarrow m
\]
must be an identity, since \(m\) receives no non-identity morphism.
Hence \(m=\iota r(m)\in\Uc\).

Let
\(
(m_A,\alpha_A\colon m_A\to d)
\)
be a representative chosen by a scaffold of \(\CC\). We have
\(m_A,d\in\Uc\), and since \(\Uc\) is full, the morphism
\(\alpha_A\) also belongs to \(\Uc\). Moreover, the connected component of
\((m_A,\alpha_A)\) in \(\partial(\Uc\downarrow d)\) corresponds to its
connected component in \(\partial(\CC\downarrow d)\). Thus, every scaffold choice in one category is the same scaffold choice in the other.
\end{proof}

The preceding proposition has the following two consequences.
\begin{proposition}\label{prop:scaffold_invariance_consequences}
Let \(\CC\) be a finite acyclic category, and let
\(\Uc\subseteq\CC\) be the full subcategory obtained in either of the
following ways:

\begin{enumerate}
    \item as the final term of a finite sequence of down beat-object
    removals
    \[
    \CC=\CC_0
    \searrow
    \CC_1
    \searrow
    \cdots
    \searrow
    \CC_n=\Uc;
    \]

    \item as the fixed subcategory
    \[
    \Uc=\operatorname{Fix}(R)
    \]
    of a right vector field
    \[
    \varepsilon\colon R\Rightarrow1_\CC.
    \]
\end{enumerate}

Then every initial scaffold of \(\CC\) is contained in \(\Uc\), and a
subcategory \(\mathcal{P}\subseteq\Uc\) is an initial scaffold of \(\CC\) if and
only if it is an initial scaffold of \(\Uc\).
\end{proposition}

\begin{proof}
In the first case, every inclusion
\[
\CC_{i+1}\hookrightarrow\CC_i
\]
is a right directed deformation retract by
Proposition~\ref{prop:beat_object_vector_field}. Hence the result follows
by applying
Proposition~\ref{prop:right_retract_preserves_scaffolds}
successively.

In the second case,
Theorem~\ref{thm:vector_field_stabilization} shows that
\(\operatorname{Fix}(R)\) is a right directed deformation retract of
\(\CC\). The result again follows from
Proposition~\ref{prop:right_retract_preserves_scaffolds}.
\end{proof}
\begin{remark}
The preceding result is orientation-sensitive. Right directed deformation retracts preserve initial scaffolds because they preserve punctured lower comma categories up to directed homotopy. A dual theory of final scaffolds can be formulated using punctured upper comma categories. We do not pursue this construction here.
\end{remark}

\subsection{Examples}

\begin{example}\label{ex:initial-scaffolds-parallel-morphisms}
For \(i=1,2,3\), let \(\CC_i\) be a category generated by
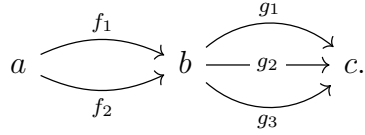
\[
\begin{tikzcd}[column sep=4em]
a
  \arrow[r, "f_1", bend left=25]
  \arrow[r, "f_2"', bend right=25]
&
b
  \arrow[r, "g_1", bend left=40]
  \arrow[r, "g_2" description]
  \arrow[r, "g_3"', bend right=40]
&
c.
\end{tikzcd}
\]
Let \(\CC_1\) be the free category on this graph. Let \(\CC_2\) be
obtained by imposing
\[
g_1f_1=g_2f_1,
\qquad
g_3f_1=g_1f_2,
\]
and let \(\CC_3\) be obtained by identifying all the composites
\(g_jf_i\).

In all three cases, \(a\) is the unique source. Moreover,
\(\partial(\CC_i\downarrow b)\) consists of the two isolated objects
\(f_1\) and \(f_2\). Hence \(b\) is lower essential, and every initial
scaffold contains both morphisms
\[
f_1,f_2\colon a\to b.
\]

In the free category \(\CC_1\), the punctured comma category
\(\partial(\CC_1\downarrow c)\) has three connected components, namely
those containing \(g_1,g_2,g_3\), respectively. Thus \(c\) is lower
essential. Choosing \(g_jf_1\) in the component containing \(g_j\)
gives an initial scaffold \(\mathcal P_1\) with
\[
\operatorname{Ob}(\mathcal P_1)=\{a,b,c\}
\]
and non-identity morphisms
\[
f_1,f_2\colon a\to b,
\qquad
g_1f_1,g_2f_1,g_3f_1\colon a\to c.
\]

In \(\CC_2\), the equality \(g_1f_1=g_2f_1\) connects the components
containing \(g_1\) and \(g_2\), while
\(g_3f_1=g_1f_2\) connects those containing \(g_3\) and \(g_1\).
Therefore, \(\partial(\CC_2\downarrow c)\) is connected. The same holds
for \(\CC_3\), where all composites are identified. Consequently,
\(c\) is not lower essential in either category, and both have the
initial scaffold
\[
a
\mathrel{\substack{\xrightarrow{f_1}\\[-1mm]\xrightarrow[f_2]{}}}
b.
\]

Thus, relations among composites may change the connected components of
punctured comma categories and hence the corresponding initial
scaffolds. Notice also that \(\CC_2\) and \(\CC_3\) have the same initial
scaffold: the construction detects the connectedness of the punctured
comma categories, but not their full categorical structure.
\end{example}

\begin{example}\label{ex:incident-database-reduction}
Let \(\CC\) be the finite acyclic category generated by
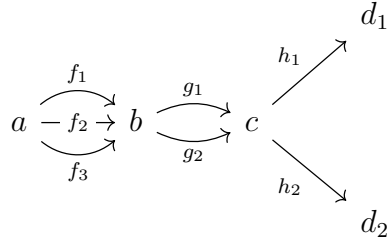
\[
\begin{tikzcd}
                                                                                               &                                                                   &                                        & d_1 \\
a \arrow[r, "f_1", bend left=40] \arrow[r, "f_2" description] \arrow[r, "f_3"', bend right=40] & b \arrow[r, "g_1", bend left=25] \arrow[r, "g_2"', bend right=25] & c \arrow[ru, "h_1"] \arrow[rd, "h_2"'] &     \\
                                                                                               &                                                                   &                                        & d_2
\end{tikzcd}
\]
subject to the relations
\[
g_1f_1=g_2f_1,
\qquad
g_1f_2=g_2f_3.
\]

We interpret \(\CC\) as a database schema for processing incident
reports. The object \(a\) represents raw reports, \(b\) processed cases,
and \(c\) decisions. The morphisms
\[
f_1,f_2,f_3\colon a\to b
\]
represent three treatment procedures, while
\[
g_1,g_2\colon b\to c
\]
represent two decision methods. Finally, \(d_1\) and \(d_2\) represent
two external centres, and \(h_1\) and \(h_2\) encode the reports sent to
them after a decision has been made. The defining relations impose
consistency between certain treatment--decision pipelines.

Define \(R\colon\CC\to\CC\) on objects by
\[
R(a)=a,\qquad R(b)=b,\qquad R(c)=c,\qquad
R(d_1)=R(d_2)=c,
\]
and on generators by
\[
R(f_i)=f_i,\qquad
R(g_j)=g_j,\qquad
R(h_1)=R(h_2)=1_c.
\]
The components
\[
\varepsilon_a=1_a,
\qquad
\varepsilon_b=1_b,
\qquad
\varepsilon_c=1_c,
\qquad
\varepsilon_{d_1}=h_1,
\qquad
\varepsilon_{d_2}=h_2
\]
define a natural transformation
\[
\varepsilon\colon R\Rightarrow1_\CC.
\]
Hence \((R,\varepsilon)\) is a right vector field. It is already
idempotent, and its fixed subcategory is
\[
\operatorname{Fix}(R)
=
\CC|_{\{a,b,c\}}.
\]
Thus, the vector field removes the two reporting centres from the reduced
schema, while their values remain determined functorially by the selected
decision.

We next compute an initial scaffold of the fixed subcategory
\[
\operatorname{Fix}(R)=\CC|_{\{a,b,c\}}.
\]
The object \(a\) is its unique source. Moreover,
\[
\partial\bigl(\operatorname{Fix}(R)\downarrow b\bigr)
\]
consists of the three isolated objects
\[
f_1,\qquad f_2,\qquad f_3.
\]
Therefore, \(b\) is lower essential, and every initial scaffold of
\(\operatorname{Fix}(R)\) contains the three morphisms
\[
f_1,f_2,f_3\colon a\to b.
\]

By contrast,
\[
\partial\bigl(\operatorname{Fix}(R)\downarrow c\bigr)
\]
is connected. Indeed, every composite \(g_jf_i\) admits a morphism to
\(g_j\), induced by \(f_i\), and the equality
\[
g_1f_1=g_2f_1
\]
connects \(g_1\) and \(g_2\). The second relation provides an additional
connection between the two branches. Consequently,
\[
M_{\operatorname{Fix}(R)}=\{a\},
\qquad
E_{\operatorname{Fix}(R)}=\{b\},
\qquad
I_{\operatorname{Fix}(R)}=\{a,b\}.
\]

Thus, an initial scaffold of \(\operatorname{Fix}(R)\) is
\[
\mathcal P
=
\left(
\begin{tikzcd}[column sep=4em, baseline=-0.5ex]
a
  \arrow[r, "f_1", bend left=40]
  \arrow[r, "f_2" description]
  \arrow[r, "f_3"', bend right=40]
&
b.
\end{tikzcd}
\right)
\]
By Proposition~\ref{prop:scaffold_invariance_consequences},
\(\mathcal P\) is also an initial scaffold of \(\CC\).

Consider now a database
\[
X\colon\CC\longrightarrow\Sets.
\]
Let
\[
X(a)=\{r_1,r_2,r_3,r_4,r_5,r_6\}
\]
be a set of raw incident reports and let
\[
X(b)=\{A,B,C,D\}
\]
be a set of processed cases. The three treatment procedures are given by
\[
\begin{array}{c|cccccc}
r
&
r_1&r_2&r_3&r_4&r_5&r_6
\\
\hline
X(f_1)(r)
&
A&A&B&B&C&C
\\
X(f_2)(r)
&
A&B&B&C&C&D
\\
X(f_3)(r)
&
A&B&B&D&C&A.
\end{array}
\]

Let
\[
X(c)=\{\mathrm{low},\mathrm{medium},\mathrm{high}\},
\]
and define the two decision methods by
\[
\begin{array}{c|cccc}
q&A&B&C&D\\
\hline
X(g_1)(q)
&
\mathrm{low}
&
\mathrm{medium}
&
\mathrm{high}
&
\mathrm{low}
\\
X(g_2)(q)
&
\mathrm{low}
&
\mathrm{medium}
&
\mathrm{high}
&
\mathrm{high}.
\end{array}
\]
These functions satisfy
\[
X(g_1)X(f_1)=X(g_2)X(f_1)
\]
and
\[
X(g_1)X(f_2)=X(g_2)X(f_3),
\]
so \(X\) respects the defining relations of \(\CC\).

Finally, let
\[
X(d_1)=\{\mathrm{routine},\mathrm{urgent}\},
\qquad
X(d_2)=\{\mathrm{archive},\mathrm{follow\mbox{-}up}\}.
\]
The reports sent to the two centres are determined by
\[
\begin{array}{c|ccc}
q
&
\mathrm{low}
&
\mathrm{medium}
&
\mathrm{high}
\\
\hline
X(h_1)(q)
&
\mathrm{routine}
&
\mathrm{routine}
&
\mathrm{urgent}
\\
X(h_2)(q)
&
\mathrm{archive}
&
\mathrm{follow\mbox{-}up}
&
\mathrm{follow\mbox{-}up}.
\end{array}
\]

The vector-field reduction and the initial scaffold give successive
natural bijections
\[
\Gamma(\CC,X)
\cong
\Gamma\bigl(
\operatorname{Fix}(R),
X|_{\operatorname{Fix}(R)}
\bigr)
\cong
\Gamma\bigl(
\mathcal P,
X|_{\mathcal P}
\bigr).
\]
A section over \(\mathcal P\) consists of a raw report
\(r\in X(a)\) and a processed case \(q\in X(b)\) satisfying
\[
X(f_1)(r)
=
X(f_2)(r)
=
X(f_3)(r)
=
q.
\]
Inspection of the table shows that this happens precisely for
\[
(r_1,A),
\qquad
(r_3,B),
\qquad
(r_5,C).
\]
Hence
\[
|\Gamma(\CC,X)|=3.
\]

The remaining values of each global section are uniquely determined by
the decision and reporting maps. The three sections are
\[
(r_1,A,\mathrm{low},\mathrm{routine},\mathrm{archive}),
\]
\[
(r_3,B,\mathrm{medium},\mathrm{routine},
\mathrm{follow\mbox{-}up}),
\]
and
\[
(r_5,C,\mathrm{high},\mathrm{urgent},
\mathrm{follow\mbox{-}up}).
\]

Thus, the right vector field first reduces the five-object schema
\(\CC\) to the three-object category \(\operatorname{Fix}(R)\). Its
initial scaffold then reduces the global-section problem to comparing the
three treatment procedures
\[
f_1,f_2,f_3\colon a\to b.
\]
Once a coherent processed case has been selected, the decision and the
two reports are recovered uniquely through \(g_1,g_2,h_1\), and \(h_2\).
\end{example}

\bibliographystyle{plain}
\bibliography{biblio}

\end{document}